\documentclass[10pt,a4paper]{amsart}

\usepackage[latin1]{inputenc}
\usepackage[leqno]{amsmath}
\usepackage{amsthm}
\usepackage[mathscr]{eucal}
\usepackage{times,verbatim}

\newtheoremstyle{theorem}{8pt\vfill}{8pt\vfill}{\itshape}{}{\bfseries}{.}{.5em}{}
\newtheoremstyle{paragraph}{8pt\vfill}{8pt\vfill}{}{}{\bfseries}{}{.5em}{}
\theoremstyle{theorem}                                       
\newtheorem{thm}{Theorem}
\newtheorem*{lemma}{Lemma}
\theoremstyle{definition}
\newtheorem*{rem}{Remark}

\DeclareMathOperator{\SL}{SL}      
\DeclareMathOperator{\GL}{GL}

\def\cG{\mathcal{G}}
\def\cX{\mathcal{X}}
\def\Z{\mathbb{Z}}

\def\C{\mathbb{C}}
\def\F{\mathbb{F}}
\def\Fun{{\F_1}}                             
\def\Gm{\mathbb{G}_m}
\def\A{\mathbb{A}}
\def\P{\mathbb{P}}

\title{Functional equations for zeta functions of $\Fun$-schemes}
\author{Oliver Lorscheid}
\address{City College of CUNY, Dept. Math., 160 Convent Ave., NYC, NY 10031, U.S.}
\email{olorscheid@ccny.cuny.edu}
\thanks{I like to thank Takashi Ono for drawing my attention to the symmetries occuring in the counting polynomials of split reductive group schemes, and I like to thank Markus Reineke for his explanations on the comparision theorem for liftable smooth varieties.}

\begin{document}

\begin{abstract}
 For a scheme $X$ whose $\F_q$-rational points are counted by a polynomial $N(q)=\sum a_iq^i$, the $\Fun$-zeta function is defined as $\zeta_\cX(s)=\prod(s-i)^{-a_i}$. Define $\chi=N(1)$. In this paper we show that if $X$ is a smooth projective scheme, then its $\Fun$-zeta function satisfies the functional equation $\zeta_\cX(n-s) = (-1)^\chi \zeta_\cX(s)$. We further show that the $\Fun$-zeta function $\zeta_\cG(s)$ of a split reductive group scheme $G$ of rank $r$ with $N$ positive roots satisfies the functional equation $\zeta_\cG(r+N-s) = (-1)^\chi  \big(\zeta_\cG(s) \big)^{(-1)^r}$.
\end{abstract}

\maketitle


\section{Introduction}

\noindent
In recent years around a dozen different suggestion of what a scheme over $\Fun$ should be appeared in literature (cf.\ \cite{LL10}). The common motivation for all these approaches is to provide a framework in which Deligne's proof of the Weyl conjectures can be transfered to characteristic $0$ in order to proof the Riemann hypothesis. Roughly speaking, $\Fun$ should be thought of as a field of coefficients for $\Z$, and $\Fun$-schemes $\cX$ should have a base extension $\cX_\Z$ to $\Z$ which is a scheme in the usual sense.

Though it is not clear yet whether one of the existing $\Fun$-geometries comes close to this goal, and thus in particular it is not clear what the appropriate notion of an $\Fun$-scheme should be, the zeta function $\zeta_\cX(s)$ of such an elusive $\Fun$-scheme $\cX$ is determined by the scheme $X=\cX_\Z$.

Namely, let $X$ be a variety of dimension $n$ over $\Z$, i.e.\ a scheme such that $X_k$ is an variety of dimension $n$ for any field $k$. Assume further that $X$ has a counting polynomial 
$$ N(q)\quad = \quad \sum_{i=0}^n a_i\ q^i\qquad \in \qquad \Z[q], $$
i.e.\ the number of $\F_q$-rational points is counted by $\#X(\F_q)=N(q)$ for every prime power $q$. If $X$ descents to an $\Fun$-scheme $\cX$, i.e.\ $\cX_\Z\simeq X$, then $\cX$ has the zeta function
$$ \zeta_\cX(s) \quad = \quad \lim_{q\to 1}\quad (q-1)^{\chi}\ \zeta_X(q,s) $$
where $\zeta_X(q,s) = \exp\left( \sum_{r\geq 1}N(q^r)q^{-sr}/r \right)$ is the zeta function of $X\otimes\F_q$ if $q$ is a prime power and $\chi=N(1)$ is the order the pole of $\zeta_X(q,s)$ in $q=1$ (cf.\ \cite{Soule04}). This expression comes down to
$$ \zeta_\cX(s) \quad = \quad \prod_{i=0}^n (s-i)^{-a_i} $$
(\cite[Lemme 1]{Soule04}). 

From this it is clear that $\zeta_\cX(s)$ is a rational function in $s$ and that its zeros (resp. poles) are at $s=i$ of order $-a_i$ for $i=0,\dotsc,n$. The only statement from the Weyl conjectures which is not obvious for zeta functions of $\Fun$-schemes is the functional equation. 

\bigskip

\section{The functional equation for smooth projective $\Fun$-schemes}

\noindent
Let $X$ be an (irreducible) smooth projective variety of dimesion $n$ with a counting polynomial $N(q)$. Let $b_0,\dotsc,b_{2n}$ be the Betti numbers of $X$, i.e. the dimensions of the singular homology groups $H_0(X_\C),\dots,H_{2n}(X_\C)$. By Poincar\'e duality, we know that $b_{2n-i}=b_i$. As a consequence of the comparision theorem for smooth liftable varieties and Deligne's proof of the Weil conjectures, we know that the counting polynomial is of the form
$$ N(q)\quad = \quad \sum_{i=0}^n b_{2i}\ q^i $$
and that $b_i=0$ if $i$ is odd (cf.\ \cite{Deitmar06} and \cite{Reineke08}). Thus $\chi=\sum_{i=0}^n b_{2i}$ is the Euler characteristic of $X_\C$ in this case (cf.\ \cite{Kurokawa05}).

Suppose $X$ has an elusive model $\cX$ over $\Fun$. Then $\cX$ has the zeta function $\zeta_\cX(s)=\prod_{i=0}^n (s-i)^{-b_{2i}}$.

\begin{thm}
 The zeta function $\zeta_\cX(s)$ satisfies the functional equation
 $$ \zeta_\cX(n-s) \quad = \quad (-1)^\chi\ \zeta_\cX(s) $$
 and the factor equals $-1$ if and only if $n$ is even and $b_n$ is odd.
\end{thm}

\begin{proof}
 We calculate
 \begin{eqnarray*}
   \zeta_\cX(n-s)  & = &   \prod_{i=0}^n ((n-s)-i)^{-b_{2i}} \\
                   & = &   \prod_{i=0}^n (-1)^{b_{2i}}(s-(n-i))^{-b_{2i}} \\
                   & = &   (-1)^\chi\ \prod_{i=0}^d (s-(n-i))^{-b_{2n-2i}} \\
 \end{eqnarray*}
 where we used $b_{2n-2i}=b_{2i}$ in the last equation. If we now substitute $i$ by $n-i$ in this expression, we obtain
 $$ \zeta_\cX(n-s) \quad = \quad (-1)^\chi\ \prod_{i=0}^n (s-i)^{-b_{2i}} \quad = \quad (-1)^\chi\ \zeta_\cX(s). $$
 If $n$ is odd, then there is an even number of non-trivial Betti numbers and $\chi=2b_0+2b_2+\dotsb+2b_{n-1}$ is even. If $n$ is odd, then $\chi=2b_0+2b_2+\dotsb+2b_{n-2}+b_n$ has the same parity as $b_n$. Thus the additional statement.
\end{proof}

\begin{rem}
 Note the similarity with the functional equation for motivic zeta functions as in \cite[Thm.\ 1]{Kahn}. Amongst other factors, also $(-1)^{\chi(M)}$ appears in the functional equation of the zeta function of a motive $M$ where $\chi(M)$ is the (positive part of the) Euler characteristic of $M$.
\end{rem}

\bigskip

\section{The functional equation for reductive groups over $\Fun$}

\noindent
The above observations imply further a functional equation for reductive group schemes over $\Fun$. Note that Soule's and Connes and Consani's approaches towards $\Fun$-geometry indeed succeeded in descending split reductive group schemes from $\Z$ to $\Fun$ (cf.\ \cite{CC08}, \cite{LL09}, \cite{Lorscheid2009}).

Let $G$ be a split reductive group scheme of rank $r$ with Borel group $B$ and maximal split torus $T\subset B$. Let $N$ be the normalizer of $T$ in $G$ and $W=N(\Z)/T(\Z)$ be the Weyl group. The Bruhat decomposition of $G$ (with respect to $T$ and $B$) is the morphism
$$ \coprod_{w\in W} BwB \quad \longrightarrow \quad G \;, $$
induced by the subscheme inclusions $BwB\to G$, which has the property that it induces a bijection between the $k$-rational points for every field $k$. We have $B\simeq \Gm^r\times\A^N$ as schemes where $N$ is the number of positive roots of $G$, and $BwB\simeq \Gm^r\times\A^{N+\lambda(w)}$ where $\lambda(w)$ is the length of $w\in W$. With this we can calculate the counting polynomial of $G$ as
$$ N(q) \quad = \quad \# \coprod_{w\in W} BwB (\F_q) \quad = \quad (q-1)^r q^N \sum_{w\in W} q^{\lambda(w)}. $$

The quotient variety $G/B$ is a smooth projective scheme of dimension $N$ with counting function $N_{G/B}(q)=\big((q-1)^{r} q^{N}\big)^{-1} N(q)=\sum_{w\in W} q^{\lambda(w)}$. Let $b_0,\dotsc,b_{2N}$ be the Betti numbers of $G/B$, then we know from the previous section that $N_{G/B}(q)=\sum_{l=0}^N b_{2l}q^l$ and that $b_{2N-2l}=b_{2l}$.

Thus we obtain for the counting polynomial of $G$ that 
\begin{eqnarray*}
 N(q) & = & q^N \ \ \Bigg(\sum_{k=0}^r (-1)^{r-k}\binom rk q^k\Bigg)\ \cdot \ \Bigg(\sum_{l=0}^N b_{2l}q^l\Bigg) \\
      & = & \sum_{i=0}^{d} \ \ \Bigg(\sum_{k+l=i-N} (-1)^{r-k}\binom rk b_{2l}\Bigg)\ \ q^i \\
\end{eqnarray*}
where $d=r+2N$ is the dimension of $G$ and with the convention that $\binom rk=0$ if $k<0$ or $k>r$. Denote by $a_i=\sum_{k+l=i-N} (-1)^{r-k}\binom rk b_{2l}$ the coeffients of $N(q)$.

\begin{lemma}
 We have $a_0=\dotsb=a_{N-1}=0$ and $a_{d-i} \ = \ (-1)^r\ a_{i+N}$.
\end{lemma}

\begin{proof}
 The first statement follows from the fact that $N(q)$ is divisible by $q^N$. For the second statement we use the symmetries $\binom rk=\binom r{r-k}$ and $b_{2N-2l}=b_{2l}$ to calculate
 \begin{eqnarray*}
  a_{d-i} & = & \sum_{k+l=d-i-N} (-1)^{r-k}\binom rk b_{2l} \\
          & = & \sum_{k+l=d-i-N} (-1)^r(-1)^k\binom r{r-k} b_{2N-2l}.
 \end{eqnarray*}
 When we substitute $k$ by $r-k$ and $l$ by $N-l$ in this equation and use $d=r+2N$, then we obtain
 $$ a_{d-i} \quad = \quad (-1)^r \sum_{k+l=(i+N)-N} (-1)^{r-k}\binom rk b_{2l}, $$
 which is the same as $(-1)^ra_{i+N}$.
\end{proof}

Suppose $G$ has an elusive model $\cG$ over $\Fun$. Then $\cG$ has the zeta function $\zeta_\cG(s)=\prod_{i=0}^n (s-i)^{-a_{i}}$. Let $\chi=N(1)=\sum_{i=0}^d a_i$.

\begin{thm}
 The zeta function $\zeta_\cG(s)$ satisfies the functional equation
 $$ \zeta_\cG(r+N-s) \quad = \quad (-1)^\chi \ \Big(\zeta_\cG(s) \Big)^{(-1)^r}. $$
\end{thm}

\begin{proof}
 We use of the previous lemma and $r+N=d-N$ to calculate that
 \begin{eqnarray*}
  \zeta_\cG(r+N-s) & = & \prod_{i=0}^n (r+N-s-i)^{-a_{i}} \\
                   & = & \prod_{i=0}^n (d-N-s-i)^{-(-1)^ra_{d-N-i}}. \\
 \end{eqnarray*}
 After substituting $i$ by $d-N-i$, we find that
 \begin{eqnarray*}
  \zeta_\cG(d-N-s) & = & \prod_{i=0}^n (i-s)^{-(-1)^ra_{i}} \\
                   & = & (-1)^{\sum a_i} \big(\prod_{i=0}^n (s-i)^{-a_{i}}\big)^{(-1)^r} \\
                   & = & (-1)^\chi \ \Big(\zeta_\cG(s) \Big)^{(-1)^r}. \hspace{3,3cm} \qedhere \hspace{-4cm} \ 
 \end{eqnarray*}
\end{proof}

\begin{rem}
 Kurokawa calculates the $\Fun$-zeta functions of $\P^n$, $\GL(n)$ and $\SL(n)$ in \cite{Kurokawa05}. One can verify the functional equation for these examples immediately.
\end{rem}

\bigskip

\bibliographystyle{plain}

\end{document}